\documentclass{article}
\usepackage{maa-monthly}

\theoremstyle{theorem}
\newtheorem{theorem}{Theorem}

\theoremstyle{theorem}
\newtheorem{prop}{Proposition}
\theoremstyle{theorem}
\newtheorem{lemma}{Lemma}
\theoremstyle{theorem}
\newtheorem{corollary}{Corollary}

\theoremstyle{definition}

\begin{document}

\title{Using Dynamical Systems to Construct Infinitely Many Primes}
\markright{Dynamical systems and primes}
\author{Andrew Granville}

\maketitle

\begin{abstract}  Euclid's proof can be reworked to construct infinitely many primes, in many different ways, using ideas from arithmetic dynamics.
\end{abstract}

\section{Constructing infinitely many primes.}   \label{sec 1}

The first known proof that there are infinitely many primes, which appears in Euclid's \emph{Elements}, is a proof by contradiction. This proof does not indicate how to find primes, despite establishing that there are infinitely many of them. However a minor variant does do so. The key to either proof is the theorem that
\begin{center}
\emph{Every integer $q>1$ has a prime factor.}
\end{center}
(This was first formally proved by Euclid on the way to establishing the fundamental theorem of arithmetic.) 
The idea is to 
 \begin{center}
\emph{Construct an infinite sequence of distinct, pairwise coprime,  integers $a_0,a_1,\ldots$},
\end{center}
that is, a sequence for which gcd$(a_m,a_n)=1$ whenever $m\ne n$. We then obtain infinitely many primes,  the prime factors of the $a_n$, as proved in the following result.

\begin{prop} Suppose that $a_0,a_1,\ldots$ is an infinite sequence of distinct, pairwise coprime, integers. 
Let $p_n$ be a prime divisor of $a_n$ whenever $|a_n|>1$.   Then the $p_n$ form an infinite sequence of distinct primes.
\end{prop}

\begin{proof}  The $p_n$ are distinct for if not, then $p_m=p_n$ for some $m\ne n$ and so
\[
p_m= \textrm{gcd} (p_m,p_n) \ \textrm{divides} \  \textrm{gcd} (a_m,a_n) =1,
\]
a contradiction. As the $a_n$ are distinct, any given value (in particular, $-1, 0$ and $1$) can be attained at most once, and therefore all but at most three of the $a_n$ have absolute value $>1$, and so have a prime factor $p_n$.
\end{proof}

To construct such a sequence by modifying  Euclid's proof, let $E_0=2$ and 
\begin{center} 
$E_{n}=E_0E_1\cdots E_{n-1}+1$ for each $n\geq 1$. 
\end{center}
If $m<n$, then
$E_m$ divides $E_0E_1\cdots E_{n-1}=E_n-1$ (as $E_m$ is one of the terms in the product) and so gcd$(E_m,E_{n})$ divides gcd$(E_n-1,E_{n})=1$, which implies that 
gcd$(E_m,E_{n})=1$.   Therefore, if $p_n$ is a prime divisor of $E_n$ for each $n\geq 0$, then $p_0,p_1,\dots$ is an  infinite sequence of distinct primes.

The \emph{Fermat numbers}, $F_0,F_1,\ldots$, defined by $F_n=2^{2^n}+1$ for each $n\geq 0$, are a  more familiar sequence of pairwise coprime integers.
Fermat had actually conjectured that the  $F_n$ are all primes. His claim starts off correct: $3, 5, 17, 257, 65537$
are all prime, but is false for $F_5=641\times 6700417$, as Euler
famously noted. It is an open question as to whether there are infinitely many primes of the form $F_n$.\footnote{The only Fermat numbers known to be primes  have $n\leq 4$. We know that the  $F_n$
are composite for $5\leq n\leq 30$ and for many other $n$ besides. It is
always a significant moment when  a Fermat number is factored for
the first time. It could be that all $F_n$ with $n>4$ are composite, or they might all be prime from some sufficiently large $n$ onwards, or some might be prime and some composite. Currently, we have no way of knowing which is  true.} Nonetheless we can prove that the $F_n$ are pairwise coprime in a similar way to the $E_n$, using the identity
\begin{center} 
$F_{n}=F_0F_1\cdots F_{n-1}+2$ for each $n\geq 1$
\end{center}
and the fact that the $F_n$ are all odd. We   deduce from Proposition 1 that 
if $p_n$ is a prime divisor of $F_n$ for each $n\geq 0$, then $p_0,p_1,\dots$ is an  infinite sequence of distinct primes.\footnote{Goldbach was the first to use the prime divisors of the Fermat numbers to prove that there are infinitely many primes,  in a letter  to Euler in late July 1730.}

For any sequence of integers $(a_n)_{n\geq 0}$ we call $p_n$ a \emph{private prime factor} of $a_n$ if it divides $a_n$ and no other $a_m$. The primes $p_n$ constructed in Proposition 1 are therefore each private prime factors of the $a_n$.

\section{A simpler formulation.}   \label{sec 2}

In Section \ref{sec 1}, the  sequences $(E_n)_{n\geq 0}$ and $(F_n)_{n\geq 0}$  were constructed in a similar  way, multiplying all of the terms of the sequence so far together and adding a constant. We would like a simpler, unified way to view these two sequences, with an eye to generalization. This is not difficult because the recurrence for the $E_n$ can be rewritten as
\[
E_{n+1}=E_0E_1\cdots E_{n-1}\cdot E_n+1 = (E_n-1)E_n+1 = f(E_n) ,
\]
where $f(x)=x^2-x+1$. Similarly the  $F_n$-values can be determined by  
\[
F_{n+1}= (2^{2^n}+1)(2^{2^n}-1)+2 = F_n(F_n-2)+1 = f(F_n) ,
\]
where $f(x)=x^2-2x+2$. So they are both examples of  sequences $(x_n)_{n\geq 0}$  for which 
\[
x_{n+1} = f(x_n)
\]
for some polynomial $f(x)\in \mathbb Z[x]$. The terms of the sequence are all given by the recursive formula, so that
\[ 
x_n = \underbrace{f( f(\ldots f}_{n \ \text{times}}(x_0))) = f^n(x_0),
\]
where the notation $f^n$ denotes the polynomial obtained by composing $f$ with itself $n$ times (which is definitely \emph{not} the $n$th power of $f$). Any such sequence $(x_n)_{n\geq 0}$ is called the \emph{orbit} of $x_0$ under the map $f$, since the sequence is completely determined once one knows $x_0$ and $f$. We  sometimes  write the orbit as
$x_0\to x_1\to x_2\to \cdots .$

The key to the proof that the $E_n$'s are pairwise coprime is that $E_n\equiv 1 \pmod {E_m}$ whenever $n>m\geq 0$; and
the key to the proof that the $F_n$'s are pairwise coprime is that $F_n\equiv 2 \pmod {F_m}$ whenever $n>m\geq 0$. 
These congruences are not difficult to   deduce using  the following lemma from elementary number theory.

\begin{lemma}\label{congs} Let $g(x)$ be a polynomial with integer coefficients. For any integers $a\ne b$, their difference $a-b$ divides $g(a)-g(b)$. Moreover, if $m$ is an integer  for which $a\equiv b \pmod m$, then $g(a)\equiv g(b)\pmod m$.
\end{lemma}

For $f(x)=x^2-x+1$, the orbit of $0$  under the map $f$ is 
$0\to 1\to 1\to \cdots .$ 
Therefore if $n>m\geq 0$, then we have
\[
E_{n}=f^{n-m}(E_m) \equiv f^{n-m}(0)=1  \pmod {E_m},
\]
and so $(E_n,E_m)=(1,E_m)=1$.
(For clarity, we have used an ``=" sign when two numbers are equal despite working with a congruence.
Typically we use the notation $(a,b)$ rather than gcd$(a,b)$.)
Similarly if $f(x)=x^2-2x+2$, then the orbit of $0$  under the map $f$ is 
$0\to 2\to 2\to \cdots$ and so  
$F_{n}=f^{n-m}(F_m) \equiv f^{n-m}(0) =2  \pmod {F_m}$, which implies that if $n>m$, then
$(F_n,F_m)=(2,F_m)=1$.

This reformulation of two of the best-known proofs of the infinitude of primes hints at the possibility of a more general approach.

\section{Different starting points.}  \label{sec 3}

We just saw that the orbit of $2$  under the map $x\to  x^2-x+1$ is an infinite sequence of pairwise coprime integers, $2\to 3\to 7\to 43\to 1807 \to \cdots .$ What about other orbits? The orbit of $4$ is also an infinite sequence of pairwise coprime integers beginning $4\to 13\to 157\to 24493\to \cdots$, as is the orbit of $5$ which begins $5\to 21 \to 421 \to \cdots .$ The same proof as before yields that no two integers in a given orbit have a common factor.

 The orbit of $3$  under the  the map $x\to  x^2-2x+2$ yielded the Fermat numbers $3\to 5\to 17\to 257 \to \cdots$ , but starting at $4$ we get $4\to 10\to 82\to 6562\to \cdots .$ These are obviously not pairwise prime as every number in the orbit is even, but if we divide through by $2$, then we get
\[ 2\to 5\to 41\to 3281\to \cdots \]
which are pairwise coprime. To prove this, note that  $x_0=4$ and $x_{n+1}= x_n^2-2x_n+2$ for all $n\geq 0$, and so the above proof yields that if $m<n$ then $(x_m,x_n)=(x_m,2)=2$. Therefore, taking 
$a_n=x_n/2$ for every $n$ we deduce that $(a_m,a_n)=(x_m/2,x_n/2) =(x_m/2,1)=1$.
This same idea works for every orbit under this map: we get an infinite sequence of pairwise coprime integers by dividing through by 1 or 2, depending on whether $x_0$ is odd or even.

However, things can be more complicated: Consider the orbit of  $x_0=3$ under the map $x\to x^2-6x-1$.  We have
\[
3 \to -10 \to 159 \to 24326 \to 591608319 \to \cdots .
\]
Here $x_n$ is divisible by $3$ if $n$ is even, and is divisible by $2$ if $n$ is odd. If we let $a_n=x_n/3$ when $n$ is even, and $a_n=x_n/2$ when $n$ is odd, then one can show the terms of the resulting sequence,
\[
1\to -5 \to 53 \to 12163 \to 197202773 \to \cdots ,
\]
are indeed pairwise coprime. 

Another surprising example is given  by the orbit   of $6$ under the map $x\to 7+x^5(x-1)(x-7)$. Reducing the elements of the orbit mod $7$ we find that
\[
6 \to 5 \to 4\to 3\to 2\to 1\to 0\to 0\to \cdots  \ \pmod 7,
\]
as $x^5(x-1)(x-7)\equiv x^6(x-1) \pmod 7$, which is $\equiv x-1 \pmod 7$ if $x\not\equiv 0 \pmod 7$ by Fermat's little theorem. So   $x_n$ is divisible by $7$ for every $n\geq 6$, but for no smaller $n$, and to obtain the pairwise coprime  $a_n$ we let $a_n=x_n$ for $n\leq 5$, and  $a_n=x_n/7$ once $n\geq 6$.

It starts to look as though it might become complicated to formulate how to define a sequence $(a_n)_{n\geq 0}$  of pairwise coprime integers in general; certainly a  case-by-case description is unappealing. However, there is a simpler way to obtain the $a_n$: In  these last two examples we have  $a_n=x_n/(x_n,6)$ and then $a_n=x_n/(x_n,7)$, respectively, for all $n\geq 0$, a description that will generalize well.

\section{Dynamical systems and the infinitude of primes.}   \label{sec 4}

One models evolution by  determining the future development of the object of study   from   its current state (more sophisticated models incorporate the possibility of random mutations). This gives rise to  \emph{dynamical systems}, a rich and bountiful area of study. One simple model is that the state of the object at time $n$ is denoted by $x_n$, and given an initial state $x_0$, one can find subsequent states via a map $x_n\to f(x_n)=x_{n+1}$ for some given function $f(.)$. 
Orbits of linear polynomials $f(.)$  are  easy  to understand,\footnote{One can verify: if $f(t)=at+b$, then $x_n=x_0+nb$ if $a=1$, and $x_n=a^nx_0+b/(1-a)$ if $a\ne 1$.} but   quadratic polynomials  can give rise to evolution that is very far from what one might naively guess (the reader might look into the extraordinary Mandelbrot set).

This is  the set-up that we had above! For us, $f(x)$ is a polynomial with integer coefficients and $x_0$  an integer. It will be useful to use  dynamical systems terminology.

If $f^n(\alpha)=\alpha$ for some integer $n\geq 1$, then (the orbit of)  $\alpha$ is \emph{periodic}, and   the smallest such  $n$ is the \emph{exact period} length for $\alpha$. The orbit  begins with  the \emph{cycle} 
\[
\alpha, f(\alpha),\ldots, f^{n-1}(\alpha)
\]
of distinct values, and   repeats itself, so that $f^n(\alpha)=\alpha, f^{n+1}(\alpha)=f(\alpha),\ldots$, and, in general  $f^{n+k}(\alpha)=f^{k}(\alpha)$ for all $k\geq 0$.

The number $\alpha$ is  \emph{preperiodic} if  $f^m(\alpha)$ is periodic for some $m\geq 0$, and is 
\emph{strictly preperiodic} if $\alpha$ is preperiodic but not itself periodic. In all of our examples so far, $0$ has been strictly preperiodic. In fact, if any two elements of the orbit of $\alpha$ are equal, say $f^{m+n}(\alpha)=f^{m}(\alpha)$,  then $f^{k+n}(\alpha)=f^{k-m}(f^{m+n}(\alpha))=f^{k-m}(f^{m}(\alpha))=f^{k}(\alpha)$ for all $k\geq m$, so that $\alpha$ is preperiodic.

Finally,  $\alpha$ has a \emph{wandering orbit} if it is not preperiodic, that is, if its orbit never repeats itself so that the $\{ f^m(\alpha)\}_{m\geq 0}$ are all distinct. Therefore, we wish to start only with integers $x_0$ that have wandering orbits.

We now state our general result for constructing    infinitely many primes from orbits of a polynomial map.

\begin{theorem}  \label{1st result}
Suppose that $f(x)\in \mathbb Z[x]$, and that $0$ is a strictly preperiodic point of the map $x\to f(x)$. 
Let $\ell(f)=\text{\rm lcm}[f(0),f^2(0)]$. 
For any integer $x_0$ that has a wandering orbit, $(x_n)_{n\geq 0}$, let
\[    a_n= \frac{x_n}{\text{\rm gcd}(x_n,\ell (f))} \text{ \ \ for all \ \ } n\geq 0.\]
The $(a_n)_{n\geq 0}$ are an infinite sequence of pairwise coprime integers, and if $n\geq 3$  then $a_n$  has a private prime factor.
\end{theorem}

For example, $\ell(f)=1,2,6$ and $7$, respectively, for the four polynomials $f(x)$ in the examples of Section \ref{sec 3}.
 
\begin{proof} Suppose that $k=n-m>0$. Then, by Lemma \ref{congs}, 
 \[
x_n = f^{k}(x_m) \equiv f^k(0) \pmod {x_m},
\]
and so gcd$(x_m,x_n)$ divides gcd$(x_m,f^{k}(0))$, which divides $f^{k}(0)$.
But this divides 
\[
L(f):= \text{\rm lcm} [ f^k(0):\ k\geq 1 ] ,
\]
which is the lcm of a finite number of nonzero integers, as $0$ is preperiodic.
Therefore $(x_m,x_n)$ divides $L(f)$, and so 
$(x_m,x_n)$ divides both $(x_m,L(f))$ and $(x_n,L(f))$. This implies that
$A_m:=x_m/(x_m,L(f))$ divides $x_m/(x_m,x_n)$, and $A_n$ divides $x_n/(x_m,x_n)$. But $x_m/(x_m,x_n)$ and $x_n/(x_m,x_n)$ are pairwise coprime which therefore implies that $(A_m,A_n)=1$.

We will show below that $L(f)=\ell(f)$, that is, the lcm of all the elements of the orbit of $0$ is the same as the lcm of the first two terms. This then implies that $a_n=A_n$ for all $n$.

One can also prove that $|a_n|>1$ for all $n\geq 3$ (see the discussion of one case in  Section \ref{app A}), and all such $n$ must have a private prime factor.
The result follows.
\end{proof}

The example $f(x)= 3-x(x-3)^2$  with $0\to 3\to 3$ has a wandering orbit  $2\to 1\to -1\to\cdots$, so that if $x_0=2$ then $x_1=1$ and $x_2=-1$. Therefore  we cannot in general improve the lower bound, $n\geq 3$, in Theorem \ref{1st result}.

We now determine all polynomials $f$  that satisfy the hypothesis of  Theorem \ref{1st result}.

\section{Polynomial maps for which $0$ is strictly preperiodic.}  \label{sec 5}

There are severe restrictions on the  possible exact periods.

\subsection{If $0$ is strictly preperiodic, then its exact period length is either one or two.}  
We have already seen the examples $x^2-x+1$ for which $0\to1\to1$, and 
$x^2-6x-1$ for which $0\to-1\to6\to -1$, where $0$ is preperiodic with period length one and two, respectively. In fact exact periods cannot be any larger:

\begin{lemma}  \label{PeriodLength}
Let $f(x)\in \mathbb Z[x]$. If the orbit of $a_0$ is periodic, then its exact period length is either one or two.
\end{lemma}

 \begin{proof}  
Let $N$ be the exact period length so that  $a_N=a_0$, and then $a_{N+k}=f^k(a_N)=f^k(a_0)=a_k$ for all $k\geq 1$.
Now assume that  $N>1$ so that $a_1\ne a_0$. Lemma \ref{congs} implies that    $a_{n+1}-a_n$ divides  $f(a_{n+1})-f(a_n)=a_{n+2}-a_{n+1}$ for all $n\geq 0$. Therefore,
\begin{center}
$a_1-a_0$ divides $a_2-a_1$, which divides $a_3-a_2,  \ldots,$ which divides $a_N-a_{N-1}=a_0-a_{N-1}$;  
and this divides $a_1-a_{N}=a_1-a_{0}$, 
\end{center}
the nonzero number we started with. We deduce that 
\[
|a_1-a_0|\leq |a_2-a_1|\leq \cdots \leq  |a_{j+1}-a_j|\leq \cdots \leq |a_1-a_0|,
\]
and so these are all equal. Therefore there must be some $j\geq 1$ for which $a_{j+1}-a_j = -(a_j-a_{j-1})$, else each 
$a_{j+1}-a_j =a_j-a_{j-1}=\cdots =a_1-a_0$ and so
\[
0=a_N-a_0 = \sum_{j=0}^{N-1} (a_{j+1}-a_j ) = \sum_{j=0}^{N-1} (a_1-a_0) = N(a_1-a_0) \ne 0,
\]
a contradiction. 
Therefore $a_{j+1}=a_{j-1}$, and we  deduce that
\[
a_2=a_{N+2}=f^{N+1-j}(a_{j+1})=f^{N+1-j}(a_{j-1})=a_N=a_0,
\]
as desired. 
\end{proof}

To determine which polynomials can be used in Theorem \ref{1st result},
we need to understand the possible orbits of $0$.

\subsection{A classification of when $0$ is a strictly preperiodic point.} 
There are just four possibilities for the period length and the preperiod length of the orbit of $0$. 
We now give examples of each type.
For some nonzero integer $a$, the orbit of $0$ is one of: \smallskip
 
 \noindent $\boxed{ 0\to a\to a\to \cdots }$ \ \ for $x\to x^2-ax+a$ \ (e.g., $(a_n)_{n\geq 0}$ and $(F_n)_{n\geq 0}$); or

 \vskip.06in
 
 \noindent $\boxed{ 0\to -a\to a\to a \to \cdots } $ \ \ for $a=1$ or $2$, with $x\to 2x^2-1$ or $x^2-2$, 
 \vskip-.1in
 \hfill respectively; or
 
  \noindent $\boxed{ 0\to -1 \to a\to -1\to  \cdots } $
 \ \ with   $f(x)=x^2-ax-1$; or

 \vskip.06in
 
 \noindent $\boxed{ 0\to 1\to 2\to -1\to 2 \to -1 \to \cdots } $
\ \ with  $f(x)=1+x+x^2-x^3$; and

\smallskip

\noindent all other possible orbits are obtained from these, by using the observation that  if $0\to a\to b \to \cdots$ is an orbit for $x\to f(x)$, then $0\to -a\to -b \to \cdots$ is an orbit for $x\to g(x)$ where $g(x)=-f(-x)$. 

It is not difficult to determine all polynomials $f$ for which $0$ has one of these orbits. For example, in the first case above we simply need to find all $f$ for which $f(0)=f(a)=a$. One example is $f_0(x)=a$ so that $f$ is another example if and only if  $f-f_0$ has roots at $0$ and $a$. That is, $f(x)-f_0(x)$ equals $x(x-a)g(x)$ for some $g(x)\in \mathbb Z[x]$. The analogous approach works in all of the other   cases.
 \medskip
 
Proving these are the only possible orbits is not difficult using Lemma \ref{congs}, but a little tedious. To give an example, let's 
determine the possible orbits of $0$ with two elements in the preperiod and period length $1$; that is, of the shape $0\to b\to a \to a\to \cdots$ for distinct nonzero integers $a$ and $b$. Now
$b=b-0$ divides $f(b)-f(0)=a-b$ and so $b$ divides $a$. Moreover, $a=a-0$ divides $f(a)-f(0)=a-b$ and so $a$ divides $b$.
Therefore $|b|=|a|$, and so $b=-a$ as $a$ and $b$ are distinct.
\medskip

We make  several deductions from analyzing the four cases of the classification:
\begin{enumerate}
\item $0$ is strictly preperiodic if and only if $f^2(0)=f^4(0)\ne 0$;
\item $L(f)=\text{\rm lcm}[f(0),f^2(0)]=:\ell(f)$, as claimed in the proof of Theorem \ref{1st result}; 
\item $x_0$ has a wandering orbit  if and only if $x_2\ne x_4$: To see this, note that if $x_0$ is periodic, then $x_{n+2}=x_n$ for all $n\geq 0$ as the period length is either one or two.
Moreover, if $x_0$ is strictly preperiodic, then $0$ is strictly preperiodic for the map $x\to  f(x+x_0)-x_0$, with orbit $0\to b_1\to b_2\to \cdots$ where $b_n=x_n-x_0$.  In  all four cases of our classification $b_2=b_4$, and so $x_2=x_4$.
\end{enumerate}

\section{Private Prime Factors when $0$ is a periodic point.} \ \label{sec 7}
We now turn our attention to maps $x\to f(x)$ for which $0$ is not strictly preperiodic, beginning with when 
$0$ is  periodic with $f$ of degree $>1$. We have seen that the period must have length one or two.

  \subsection{When $0$ is periodic with period length 1.}\label{f period 1} 
We must have $f(x)=x^rg(x)$ for some  $g(x)\in \mathbb Z[x]$ with $g(0)\ne 0$, and $r\geq 1$.
We will assume that $g(x)$ has degree $\geq 1$,
else the integers in the orbit of $x_0$ contain only the prime factors of $x_0$ and of  $g(0)$, so do not help us to construct infinitely many primes. 

\begin{theorem} Suppose that $f(x)=x^rg(x)$ for some nonconstant $g(x)\in \mathbb Z[x]$ with $g(0)\ne 0$, and $r\geq 1$. For any given wandering orbit $(x_n)_{n\geq 0}$ with $x_0\in \mathbb Z$, define
\[
a_{n+1}:=\frac{g(x_n)}{\text{gcd}(g(x_n),g(0))}  \text{  for all  } n\geq 0.
\]
The $(a_n)_{n\geq 0}$ are an infinite sequence of pairwise coprime integers and, once $n$ is sufficiently large, each  $a_n$   has a private prime factor.
\end{theorem}

\begin{proof}
Given a wandering orbit $(x_n)_{n\geq 0}$ we see that $x_m$ divides $x_{m+1}=x_m^r g(x_m)$ for all $m$; and so $x_m$ divides $x_n$ whenever $m\leq n$. Now $a_{n+1}$ is a divisor of $x_{n+1}$ for all $n\geq 0$. If $m<n$, then 
\begin{center}
$(g(x_m),g(x_n))$ divides $(x_{m+1},g(x_n))$, which divides $(x_n,g(x_n))$, \\
which divides $(x_n,g(0))$, which divides $g(0)$.
\end{center}
Therefore $(a_{m+1},a_{n+1})=1$, that is, the $(a_n)_{n\geq 1}$ are pairwise coprime, and so $a_n$ has a  private prime factor provided $|a_n|>1$. Now $a_n\ne 0$, else $x_{n+1}=0$, which would contradict the assumption that the orbit of $x_0$ is wandering. If $|a_n|=1$,  then $g(x_n)$ is a divisor $d$, of $g(0)$, that is, $x_n$ must be a root of the polynomial
 \[
 \prod_{\substack{ d \ \text{is a divisor of } g(0)}} (g(x)-d) .
 \]
But this has finitely many roots so, as the orbit of $x_0$ is wandering, there can only be finitely many   $n$ for which $|a_n|=1$. Therefore $a_n$ must have a private prime factor for all sufficiently large $n$.
 \end{proof}

 \subsection{When $0$ is periodic with period length 2.}\label{f period 2}
We rule out the polynomials $f(x)=a-x$ for any  constant $a$, else there are no wandering orbits. 

\begin{theorem} Suppose that $f(x)\in \mathbb Z[x]$ with   $f(x)+x$ nonconstant, such that 
$0$ is periodic under the map $x\to f(x)$, with period length 2. We write $f^2(x)=x^rG(x)$ with $G(0)\ne 0$ and $r\geq 1$.
For any given wandering orbit $(x_n)_{n\geq 0}$ with $x_0\in \mathbb Z$, define
\[
a_{n+2}:=\frac{G(x_n)}{\text{gcd}(G(x_n),G(0)f(0))}  \text{  for all  } n\geq 0.
\]
The $(a_n)_{n\geq 0}$ are an infinite sequence of pairwise coprime integers and, once $n$ is sufficiently large, each  $a_n$   has a private prime factor.
\end{theorem}

\begin{proof} We begin by showing that $G(x)$ cannot be a constant, $c$. If so, then $f^2(x)=cx^r$ but $f(0)\ne 0$. Therefore, $f(x)$ has just one root, as $f^2(x)$ has at least as many distinct roots as $f$. Writing
$f(x)=c(x-a)^d$ with $a\ne 0$, the roots of $f^2(x)$ are the roots of $c(x-a)^d-a$, each repeated $d$ times. This implies that $d=1$ as the roots of $c(x-a)^d-a$ are distinct, and so
$f^2(c)=c^2x-ac(c+1)$ and so $c=-1$; that is, $f(x)=a-x$, which we have already ruled out.

Now, by definition, $a_{n+2}$ is a divisor of $x_{n+2}$ for all $n\geq 0$. Suppose that $n>m\geq 0$. If $n-m$ is even, then $(G(x_m),G(x_n))$ divides $G(0)$, by the argument of the previous subsection, and so $(a_{m+2},a_{n+2})=1$.
If  $n-m$ is odd, then    $G(x_m)$,  divides $x_{m+2}$ and so divides $x_{n+1}$.
On the other hand,    $G(x_n)$ divides $x_{n+2}=f(x_{n+1})$, and so $(G(x_m),G(x_m))$ divides
$(x_{n+1}, f(x_{n+1}))=(x_{n+1}, f(0))$ which divides $f(0)$, and so $(a_{m+2},a_{n+2})=1$.

Therefore the $(a_n)_{n\geq 2}$ are pairwise coprime, and so $a_n$ has a  private prime factor provided $|a_n|>1$. We then show, just as in the previous subsection, that $a_n\ne 0$, and  there are only finitely many $n$ for which $|a_n|= 1$.
\end{proof}

We can combine the first three theorems to obtain the following result:

\begin{corollary} Suppose that $f(x)\in \mathbb Z[x]$  is not  of the form $cx^d$ or $a-x$, for any constants $a$ or $c$, and  that $0$ is preperiodic for the map $x\to f(x)$. 
Then, for any given wandering orbit $(x_n)_{n\geq 0}$ with $x_0\in \mathbb Z$, each $x_n$ has a private prime factor for all sufficiently large $n$.
\end{corollary}

We now turn our attention to the more difficult case of those $f$ for which $0$ itself has a wandering orbit.

\section{Polynomial maps where $0$ has a wandering orbit}  \label{sec 8}
Henceforth we will suppose that $0$ has a wandering orbit under the map $x\to f(x)$, for some polynomial $f(x)\in \mathbb Z[x]$. We will investigate whether the terms of wandering integer orbits all have private prime factors from some point onwards. Where better to begin than with the orbit that begins with $0$?

\subsection{Orbits that begin with $0$ are divisibility sequences.} If $n=m+k$, then
$x_n=f^k(x_m)\equiv f^k(0)=x_k \pmod {x_m}$, and so $(x_n,x_m)=(x_k,x_m)$. Subtracting $m$ several times from $n$,  we deduce that if $n=mq+r$ then $(x_n,x_m)=(x_r,x_m)$, which mirrors that $(n,m)=(r,m)$, the key step 
 in the Euclidean algorithm. Therefore, by applying the same steps as in the Euclidean algorithm we deduce that 
\[
(x_m,x_n) = x_{(m,n)} .
\]
Therefore $(x_n)_{n\geq 1}$ is a \emph{strong divisibility sequence}; for example, $x_n=2^n-1$, the orbit of $0$ under the map  $x\to f(x)=2x+1$. 
Other examples of strong divisibility sequences include the Fibonacci numbers and any second order linear recurrence sequence starting with $0$, and the denominators of multiples of a given point on an elliptic curve.

A divisibility sequence cannot have  private prime factors, because if prime $p$ divides $x_m$ then it divides $x_n$ for every $n$ divisible by $m$. Instead one looks for \emph{primitive prime factors}, that is, primes $p_n$ that divide $x_n$ but do not divide $x_m$ for any $m,\ 1\leq m<n$. It is known that for any second order linear recurrence sequence beginning with a $0$, and for any elliptic divisibility sequence, $x_n$ has a primitive prime factor for all sufficiently large $n$. The two key ideas in the proof are to show that the $x_n$ are fast-growing and that  if  $p$ is prime factor of $x_n$, but not a primitive prime factor, then $p$ but not $p^2$ divides $x_n/x_{n/p}$. In particular, one deduces that $x_n$ is too large for all its prime power factors to be powers of imprimitive primes.

In the example $x_n=2^n-3$ (the orbit of $-2$ under the map $x\to 2x+3$), and for general wandering orbits, one cannot be so precise about the power to which imprimitive primes $p$ divide $x_n$, and so it is an open question as to whether the $x_n$ all have  primitive prime factors for sufficiently large $n$. Nonetheless, we expect that this is true and, in the next subsection, we will consider this question following roughly the same outline.

\subsection{The growth of the $x_n$ and $(x_m,x_n)$.} The $x_n$ grow very  fast, and we can gain a quite precise understanding:\ If $f(x)\in \mathbb Z[x]$ has degree $d>1$ and  if $x_0$ is an integer whose orbit is wandering, then there exist real numbers $1\geq \alpha>0$ and $\beta$, which depend only on $f$, for which
 \begin{equation} \label{aAsymp}
     |x_n|\  \text{is the integer nearest to} \  \alpha \tau^{d^n} +\beta,  \text{ when} \ n\  \text{is sufficiently large},
 \end{equation}
 where $\tau>1$ is a constant that depends on both $f$ and $x_0$.\footnote{More generally, the \emph{height}, $h(r)$, of a rational number $r=p/q$ is given by $\max\{ |p|,|q|\}$; this equals $|r|$ if $r$ is a nonzero integer. For any orbit  $\{ x_n\}_{n\geq 0}$ of a rational 
function  $\phi(t)=f(t)/g(t)$, with $f(t),g(t) \in \mathbb Z[t]$  of degree $d$, there exists a constant $\tau>0$ (the exponential of the \textsl{dynamical canonical height}, as in \cite{CS} or Section 3.4 of \cite{Si}), and constants $c_1,c_2>0$ such that $c_1\tau^{d^n} < h(x_n) < c_2\tau^{d^n}$.}
  (We will sketch a proof of \eqref{aAsymp} in Section \ref{sec 9}.)
In the most interesting case, the orbit of $0$, we write $\tau=\tau_0$. 

Since $\alpha\leq 1$,  \eqref{aAsymp} implies that $|x_r|\leq 2\tau^{d^r}$ 
and $|f^r(0)|\leq 2\tau_0^{d^r}$ if $r$ is sufficiently large. 
Suppose that $0\leq m\leq n-1$. Since $(x_m,x_n)$ divides $f^{n-m}(0)$ as well as $|x_m|$, we deduce that   
\begin{equation} \label{bAsymp}
(x_m,x_n)\leq \min\{ |x_m|, f^{n-m}(0) \} \leq   2 \tau_*^{d^{n/2}}
 \end{equation}
if $n$ is sufficiently large, where $\tau_*=\max\{ \tau, \tau_0\}$. 
 
We can therefore deduce that the $x_n$ grow so fast that the product of all of the gcd$(x_m,x_n)$ with $m<n$ is far smaller than $x_n$, so it seems very likely that $x_n$ has primitive prime factors. However, this argument cannot be made into a proof since some of the primes that divide the $x_m$ with $m<n$ might divide $x_n$ to an extraordinarily high power.

  \subsection{Ruling out too many large prime power divisors.}\label{abc}  
  The $abc$ conjecture can be used  to bound solutions to equations in which all of the variables are divisible by surprisingly high powers of primes. The idea is to study pairwise coprime integer solutions to the equation
  \[ a+b=c. \]
 The \emph{$abc$-conjecture} states, in a quantitative form,  that $a,b$ and $c$ cannot all be divisible only by large powers of primes.\footnote{Mochizuki has recently  claimed to have  proved the $abc$-conjecture, but the experts have as yet struggled to agree, definitively, that Mochizuki's extraordinary ideas are all correct.}  
 For example if we have a putative solution to Fermat's last theorem, like $x^n+y^n=z^n$ we can take $a=x^n, b=y^n$ and $c=z^n$, and the $abc$-conjecture should rule out any possible solution, at least once the numbers involved are sufficiently large. One does so by bounding the size of $a,b$, and $c$ in terms of the product of the distinct primes that divide $a,b$, and $c$.  To put this  in quantitative form requires some $\epsilon$'s and that sort of thing.
 
 Fix $\epsilon>0$ (which should be thought of as being small). There exists a constant $\kappa_\epsilon>0$ (which depends on $\epsilon$, and chosen so it works for all examples), such that if $a$ and $b$ are coprime positive integers with $c=a+b$, then
 \[
   c \leq \kappa_\epsilon \left(  \prod_{\substack{p \ \text{prime} \\ p \ \text{divides} \ abc}} p \right)^{1+\epsilon} .
 \]
 This is the \emph{$abc$-conjecture}, one of the great questions of modern mathematics.
In the special case that $b=1$ we have $c=a+1$, and so
 \[
   a \leq \kappa_\epsilon \left(  \prod_{\substack{p \ \text{prime} \\ p \ \text{divides} \ a(a+1)}} p \right)^{1+\epsilon} .
 \]
 The product here is over all primes $p$ that divide the polynomial $f(x)=x^2+x$ evaluated at $x=a$. Langevin \cite{Lan} proved that a far-reaching generalization, which applies to the prime divisors of any given polynomial, is actually  a consequence of the $abc$-conjecture.  
\medskip

\noindent \textbf{Consequence of the $abc$-conjecture}: Suppose that $f(x)\in \mathbb Z[x]$ has at least $d$ distinct roots, for some $d\geq 2$. For any fixed $\epsilon>0$, there exists a constant $\kappa_{\epsilon, f}>0$, which depends on both $\epsilon$ and $f$, such that for any integer $a$ we have
\[
   |a|  \leq \kappa_{\epsilon, f}  \left(  \prod_{\substack{p \ \text{prime} \\ p \ \text{divides} \ f(a)}} p \right)^{1/(d-1)+\epsilon} .
 \]           
\medskip

\noindent We will apply this to our polynomial $f$ with $\epsilon=1$, writing $\kappa$ for $\kappa_{1, f}$ for $x_n=f(x_{n-1})$. Therefore, if $f$ has at least two distinct roots, then the $abc$-conjecture implies that 
\begin{equation} \label{abc-consequence}
 |x_{n-1}| \leq \kappa   \left(  \prod_{\substack{p \ \text{prime} \\ p \ \text{divides} \ x_n}} p \right)^{2} .
\end{equation}

\subsection{The product of the primes if $x_n$ has no new prime factors.}  If every prime factor of $x_n$ already divides a previous term in the sequence, then we have that 
\[
 \prod_{\substack{p \ \text{prime} \\ p \ \text{divides} \ x_n}} p \ \ \text{divides} \ \ 
 \prod_{m=0}^{n-1}   \prod_{\substack{p \ \text{prime} \\ p \ \text{divides} \ (x_m,x_n)}} p,\ \ \text{which divides} \ \ 
 \prod_{m=0}^{n-1}  (x_m,x_n).
 \]
Taking the product of \eqref{bAsymp} over all $m,\ 0\leq m\leq n-1$, we deduce that 
\[
 \prod_{\substack{p \ \text{prime} \\ p \ \text{divides} \ x_n}} p \leq \prod_{m=0}^{n-1}  (x_m,x_n) \leq 
 2^n \tau_*^{nd^{n/2}}
 \]
once $n$ is sufficiently large. Substituting this into \eqref{abc-consequence}, gives\footnote{An analogous argument is used to bound the imprimitive prime factors of recurrence sequences, and in a related context in \cite{IS}.}
\[
 |x_{n-1}| \leq \kappa   2^{2n} \tau_*^{2nd^{n/2}} .
\]
Finally, comparing this to \eqref{aAsymp}, taking logarithms, and dividing through by $\log \tau$, we see that there exists a constant $\gamma>0$ (which can be given explicitly in terms of $\tau$ and $\tau_*$) such that, for large enough $n$,
\[
d^{n} \leq \gamma  n d^{n/2}   .
\]
This is false once $n$ is sufficiently large, since $d^n$ grows much faster than $nd^{n/2}$  as $n$ increases. Therefore we have proved the following conditional result:

\begin{theorem} Assume the $abc$-conjecture. Suppose that $f(x)\in \mathbb Z[x]$ has at least two distinct roots, and that $0$ and $x_0$ both have wandering orbits under the map  $x\to f(x)$. Then
$x_n$ has a primitive prime factor for all sufficiently large $n$.
\end{theorem}

Note that this includes the possibility that $x_0=0$; and so $f^n(0)$ has a primitive prime factor for all sufficiently large $n$, assuming the $abc$-conjecture.


 

\subsection{Primitive prime factors of rational functions.} 
Now suppose that $\phi(t)=f(t)/g(t)$, with $f(t),g(t) \in \mathbb Z[t]$.
Beginning with a rational number $r_0$, we define $r_{n+1}=\phi(r_{n})$ for all $n\geq 0$, and write $r_n=a_n/b_n$ where
gcd$(a_n,b_n)=1$ and $b_n\geq 1$.

Suppose that $f(0)=0$ but $f(t)$ is not of the form $at^d$.  Ingram and Silverman \cite{IS} showed that if the orbit of $r_0$ is not eventually periodic, then $a_n$ has a primitive prime factor for all sufficiently large $n$.  It is believed that some result of this sort should hold whether or not $f(0)=0$, and indeed this has been confirmed by Gratton,   Nguyen, and Tucker (\cite{GNT}) under the assumption of the  $abc$-conjecture.

 \section{Asymptotic growth of wandering orbits.}  \label{sec 9}

Fix an integer $d>1$.  Let $(y_n)_{n\geq 0}$ be any sequence of positive integers for which 
\[ y_n^d<y_{n+1}<(y_n+1)^d \text{  for all } n\geq 0.\] 
If $\ell_n=y_n^{1/d^n}$ and $u_n=(y_n+1)^{1/d^n}$,
then $\ell_n<\ell_{n+1}<\cdots < u_{n+1}\leq u_n$. The $\ell_n$ are therefore an increasing bounded sequence, so must   tend to a limit, call it $\tau$. Then $\ell_n<\tau\leq u_n$ for all $n$ and so $y_n<\tau^{d^n}\leq y_n+1$, with equality in the upper bound  if and only if $y_n=\tau^{d^n}-1$ for all $n$. Therefore
 $y_n$ is the largest integer $< \tau^{d^n}$. In fact   $y_n=\lfloor \tau^{d^n}\rfloor$,  other than in the special case in which $\tau$ is an integer and  $y_n=\tau^{d^n}-1$ for all $n$.   
 
 Mills \cite{Mil} observed, in 1947, that advances in our understanding of the distribution of primes imply  that if $N$ is sufficiently large, then there is a prime between $N^3$ and $(N+1)^3$. So taking $p_0=y_0$ to be any prime greater than $N$, and then selecting  $p_{n+1}$ to be any prime in $(p_n^3,(p_n+1)^3)$ for each $n\geq 0$,
 one obtains a \emph{Mills' constant}, $\tau$, such that
 \begin{center}
$ \lfloor  \tau^{3^n}\rfloor $ is the prime $p_n$ for every integer $n\geq 0$.
 \end{center}
This proof yields   a lot of flexibility in constructing $\tau$ with this property (as we expect many primes in any interval of the form $(N^3,(N+1)^3)$), so it seems unlikely that these numbers $\tau$ are anything special beyond the beautiful property identified by Mills.

\begin{proof}[Sketch of a modification to prove \eqref{aAsymp}] For any polynomial $f(x)=ax^d+bx^{d-1}+\cdots$ with $a>0$, the polynomial 
\[ \alpha^{-1} (f(x)-\beta)- ( \alpha^{-1} (x-\beta))^d , \]
where $\alpha=a^{-1/(d-1)}$ and $\beta=-b/(da)$,  has degree $\leq d-2$. Given an integer orbit $(x_n)_{n\geq 0}$, let $y_n=\alpha^{-1} (x_n-\beta)$, so that 
$y_{n+1}-y_n^d$ can be written as a polynomial in $y_n$ of degree $\leq d-2$ (where the coefficients of the polynomial are independent of $n$). This implies that, for any $\epsilon>0$, if $n$ is sufficiently large, then
\[  (y_n-\epsilon)^d +\epsilon< y_{n+1} <  (y_n+\epsilon)^d -\epsilon .\]
One can then modify the proof above to show that $(y_n-\epsilon)^{1/d^n}$ tends to a limit $\tau$ and that 
$ y_n -\epsilon < \tau^{d^n}  < y_n +\epsilon$.   This implies that
\[ x_n=\alpha y_n+\beta\in \left(\alpha \tau^{d^n}+\beta -\frac 12,  \alpha \tau^{d^n}+\beta +\frac 12\right) \]
 taking $\epsilon=1/2\alpha$, and so $x_n$ is the nearest integer to $ \alpha \tau^{d^n}+\beta $ as claimed.
 \end{proof}

There is a power series $P(T)=  c_1 T + c_0 + \frac{c_{-1}}{T} + \frac{c_{-2}}{T^2} +\cdots$
 which satisfies the \textsl{functional equation}
 \[   P( x^d)  = f(P(x)) . \]
 To determine the $c_i$, we compare the coefficients on both sides of the equation: Comparing the coefficients of $T^d$ implies that $c_1=\alpha$;  comparing the  coefficients of $T^{d-1}$ implies that $c_0=\beta$. For each $j\geq 0$, the coefficients of $T^{d-1-j}$ involve polynomials in the  $c_i$ with $i>-j$ as well as a linear term in $c_{-j}$, which allows us to determine $c_{-j}$.

Assuming that $P(T)$ absolutely converges  once $|T|$ is sufficiently large,    if $x_m$ is large enough there is a unique positive real  $t$ for which $P(t)=x_m$. Then
 \[
 P(t^d)=f(P(t)) = f(x_m) = x_{m+1},
 \]
 and $P(t^{d^2})=f(P(t^2)) = f(x_{m+1}) = x_{m+2}$, etc.,
so that $P(t^{d^k})= x_{m+k}$ for all $k\geq 0$. Taking $\tau=t^{1/d^m}$ we deduce that 
\[
x_n = P(\tau^{d^n})
\]
for all sufficiently large $n$. 

One can sidestep the convergence issue by working with finite length truncations of $P(T)$. Thus if
$P_k(t)=\sum_{j=1-k}^1 c_jt_j$, one can prove there exists a constant $C_k>0$ such that 
\[ |x_n-P_k(\tau^{d^n})|\leq C_k / \tau^{kd^n}\] for all  $n\geq 0.$

  \subsection{The mysteries of the numbers $\tau$?}
 Although this last subsection allows one to show the existence of the constant   $\tau$, it gives no hint as to what kind of number $\tau$ is. The $\tau$ corresponding to the Fermat numbers is $2$. But what kind of number is $\tau$ for Euclid's sequence $(E_n)_{n\geq 0}$?  Is it rational? Irrational? Related to $\pi$? I have no idea.  Also, in general, is there anything special about the numbers $\tau_0$, the $\tau$-values that arise out of the wandering orbits of $0$?
 For a given $f(x)$, what does the set of $\tau$-values look like as we vary over all wandering orbits? Is each value isolated, or does the set have limit points?  If we also vary over all $f$ of given degree $d$, we have a countable subset of $\mathbb R_{>1}$. Are the $\tau$ dense in $\mathbb R_{>1}$? Or are there holes in the spectrum? If so, why? 
 Some of these questions are addressed in  section 3.4 of \cite{Si}, and there are some particularly interesting results and conjectures  there on lower bounds for  values of $\tau$ that are $>1$ in terms of the given map and the field of definition.
Gaining a better understanding of the numbers $\tau$ seems like a subject that is ripe for further and broader study.

\section{Appendix: Private prime factors for all $n\geq 3$.}  \label{app A}
One can improve Theorem  \ref{1st result} by noting that if $|a_n|=1$ (that is, $x_n$ is a divisor of $L(f)$)  then $n=0, 1$, or $2$. Our proof is given by  a tedious   case-by-case analysis. For example, when $0\to a \to a$ with $a>0$ we proceed as follows.

Suppose that $x_{n+2}=d$ is a divisor of $a$, for some  $n\geq 0$,  so that $x_n\to x_{n+1}\to d$ where $x_n,x_{n+1},x_{n+2}=d,a$ are distinct. Now $x(x-a)$ divides $f(x)-a$ and so $x_{n+1}(a-x_{n+1})$ divides $a-x_{n+2}=a-d$, and $x_n(a-x_n)$ divides $x_{n+1}-a$. Moreover, as $d$ divides $a$  and $d\ne a$ we have $1\leq a-d\leq 2a$.

Let $T=\min\{ x_{n+1}, a-x_{n+1}\}$ so that $T(a-T)=x_{n+1}(a-x_{n+1})$ divides $a-d$. Now $T\geq -1$, else $2(a+2)\leq |T(a-T)|\leq 2a$. If $T=-1$, then $a+1$ divides $a-d$, so that $d=-1$, and therefore $x_{n+1}=a+1$ as $x_{n+1}=T$ or $a-T$ and $x_{n+1}\ne d$. Finally, 
$x_n(a-x_n)$ divides $x_{n+1}-a=1$ and so $a=2$ and $x_n=1$. That is, $a=2$ and we have $1\to 3\to -1$.

We may now assume that $T\geq 1$, so that $a-1\leq |x_n(a-x_n)| = a-x_{n+1}\leq a-T\leq a-1$. This implies that $x_{n+1}=T=1$ and $x_n=a-1$ as $x_n\ne x_{n+1}$, and $a\geq 2$ as $x_n\ne 0,x_{n+1}$. Moreover, $a-1$ divides $a-d$  and $d\ne  1$ or $a$, so $d=1+k(1-a)$ for some $k\geq 1$. 
If $k\geq 2$, then $|1+k(1-a)|>a$ unless $k=2$ and $a=3$, which gives the example $2\to 1 \to -3$.
For $k=1$, we have $2-a$ divides $a$ only if $2-a$ divides $2$ and so $a=3$ or $4$.   The $a=3$ case gives rise to the example $2\to 1 \to -1$. The $a=4$ case does not work, else $2=x_n-x_{n+1}$ would divide $f(x_n)-f(x_{n+1})=x_{n+1}-x_{n+2}  =x_{n+1}-d=3$.
 
 Now $n$ must be $0$ in all of these three examples else  $x_{n-1}\to x_n\to x_{n+1} \to d$ for some $n\geq 1$.
 But $a=3$ with $x_{n-1}\to 2\to 1$ is impossible else  $x_{n-1}-2$ divides $2-1=1$, so that $x_{n-1}=t$ or $x_{n-1}=a$.
 And  $a=2$ with $x_{n-1}\to 1\to 3\to -1$ is impossible else $x_{n-1}-1$ divides $3-1=2$ and $x_{n-1}-3$ divides $1-(-1)=2$, which together imply that $x_{n-1}=2=a$.   
  \hfill \qed

\begin{acknowledgment}{Acknowledgments.} This article was first conceived of in 1997 following invigorating conversations about arithmetic dynamics with Mike Zieve. My understanding of arithmetic dynamics subsequently greatly benefitted from discussions with Jeff Lagarias and Xander Faber. Thanks to them all. I would also like to thank the referees for their many helpful suggestions. \end{acknowledgment}



  \begin{biog} \item[Andrew Granville] (and.granville@gmail.com) specializes in understanding the distribution of primes. He is co-author of the soon-to-appear graphic novel, \emph{MSI: Anatomy and Permutations} (P.U. Press)\smallskip

 \begin{affil}
    D\'epartement de math\'ematiques et de statistique, Universit\'e de Montr\'eal, CP 6128 succ. Centre-Ville, Montr\'eal, QC H3C 3J7, Canada\\
     andrew@dms.umontreal.ca;\ and \\
   Department of Mathematics, University College London,  Gower Street, London, WC1E 6BT, United Kingdom \\   a.granville@ucl.ac.uk  \end{affil} \end{biog} \vfill\eject

\end{document}